\documentclass[letterpaper,final]{siamltex}
\usepackage{amsfonts,amsmath,amssymb}
\usepackage{graphicx}
\usepackage[ruled,titlenotnumbered,nofillcomment]{algorithm2e}
\usepackage{multirow,stmaryrd}
\usepackage[usenames]{color}
\usepackage{epsfig}
\usepackage{epsf}
\usepackage{subfigure}

\newcommand{\beq}{\begin{equation}}
\newcommand{\eeq}{\end{equation}}
\newcommand{\beqs}{\begin{equation*}}
\newcommand{\eeqs}{\end{equation*}}

\newcommand{\one}{{\bf 1}}

\newcommand{\mU}{\mathcal{U}}
\newcommand{\mV}{\mathcal{V}}

\newcommand{\BC}{\begin{center}}
\newcommand{\EC}{\end{center}}
\newcommand{\bdm}{\begin{displaymath}}
\newcommand{\edm}{\end{displaymath}}

\newcommand{\dia}[1]{\text{diag}(#1)}
\newcommand{\eq}[1]{Eq.\ (\ref{eq:#1})}
\newcommand{\noeq}[1]{(\ref{eq:#1})}
\newcommand{\eqs}[1]{Eqs.\ (\ref{eq:#1})}
\newcommand{\fig}[1]{Fig.\ \ref{fig:#1}}
\newcommand{\alg}[1]{Algorithm \ref{alg:#1}}
\newcommand{\sect}[1]{Sec.\ \ref{sec:#1}}
\newcommand{\subsec}[1]{Sec.\ \ref{subsec:#1}}
\newcommand{\thm}[1]{Theorem \ref{thm:#1}}
\newcommand{\defi}[1]{Definition \ref{def:#1}}
\newcommand{\tabl}[1]{Table \ref{tab:#1}}

\newcommand{\barr}{\begin{array}}
\newcommand{\earr}{\end{array}}

\newcommand{\beqas}{\begin{eqnarray*}}
\newcommand{\eeqas}{\end{eqnarray*}}

\title{A Self-learning Algebraic Multigrid Method for Extremal Singular Triplets and Eigenpairs}

\author{
Hans De Sterck\footnotemark[1] \footnotemark[4] 
}

\begin{document}
\maketitle
\renewcommand{\thefootnote}{\fnsymbol{footnote}}
\footnotetext[1]{Department of Applied Mathematics, University of Waterloo,
Waterloo, Ontario, Canada}
\footnotetext[4]{hdesterck@uwaterloo.ca}
\begin{abstract}
A self-learning algebraic multigrid method for dominant and minimal singular triplets and eigenpairs is described.
The method consists of two multilevel phases.
In the first, multiplicative phase (setup phase), tentative singular triplets are calculated along with a multigrid hierarchy of interpolation operators that approximately fit the tentative singular vectors in a collective and self-learning manner, using multiplicative update formulas. In the second, additive phase (solve phase), the tentative singular triplets are improved up to the desired accuracy by using an additive correction scheme with fixed interpolation operators, combined with a Ritz update. A suitable generalization of the singular value decomposition is formulated that applies to the coarse levels of the multilevel cycles. The proposed algorithm combines and extends two existing multigrid approaches for symmetric positive definite eigenvalue problems to the case of dominant and minimal singular triplets.
Numerical tests on model problems from different areas show that the algorithm converges to high accuracy in a modest number of iterations, and is flexible enough to deal with a variety of problems due to its self-learning properties.
\end{abstract}
\pagestyle{myheadings}
\thispagestyle{plain}
\markboth{}{}

\section{Introduction}
In this paper we present an algebraic multigrid (AMG) method for accurately computing a few of the largest or smallest singular values and associated singular vectors of a sparse rectangular matrix $A \in I\!\!R^{m \times n}$. Let the singular value decomposition (SVD) of $A$ be given by
\begin{align}
A=U \, \Sigma \, V^t.
\end{align}
Here, $U \in I\!\!R^{m \times m}$ and $V \in I\!\!R^{n \times n}$, with
$U^t\, U=I_m$ and $V^t\, V=I_n$, where $I_m$ and $I_n$ are the unit matrices of sizes $m \times m$ and 
$n \times n$, respectively.
Matrix $\Sigma \in I\!\!R^{m \times n}$ has the $l=\min(m,n)$ singular values
$\sigma_1 \ge \sigma_2 \ge \ldots \ge \sigma_l \ge 0$ of $A$ on its diagonal.
In what follows we will normally assume that $m\ge n$, except where noted otherwise.
The columns $u_j$ of $U$ are called the left singular vectors of $A$, and the columns $v_j$ of $V$ are its right singular vectors. The $n$ singular triplets $(\sigma_j,u_j,v_j)$, $j=1, \ldots, n$, satisfy
\begin{alignat}{1}
A \, v_j &= \sigma_j \, u_j, \nonumber\\
 A^t \, u_j &= \sigma_j \, v_j.
\label{eq:f}
\end{alignat}
For the special case that $A$ is square and symmetric positive definite (SPD), the SVD of $A$ coincides with the eigendecomposition of $A$, and a suitably simplified version of the AMG method we propose in this paper will be applicable to the problem of computing a few of the largest or smallest eigenvalues and associated eigenvectors of an SPD matrix $A$.

For definiteness, we will frame the presentation in most of the paper in terms of calculating a few of the singular triplets with largest singular values (which we call dominant triplets), and we will comment on the case of the singular triplets with the smallest singular values (which we call minimal triplets) at the end of the algorithm presentation.
So we assume we seek the $n_b$ dominant singular triplets $(\sigma_j,u_j,v_j)$, $j=1, \ldots, n_b,$ of $A$, with singular values $\sigma_1 \ge \sigma_2 \ge \ldots \ge \sigma_{n_b}$.

There are many applications in scientific computing where dominant or minimal singular triplets of large sparse matrices need to be computed, see, for example, the discussion and references in \cite{kokio,baglama}. We mention a few examples. Latent semantic indexing determines concepts in documents by calculating dominant singular triplets of term-document matrices \cite{lsi}. Similarly, principal component analysis is used in exploratory data analysis to identify orthogonal components with maximal variance, which correspond to dominant singular triplets of the data matrix \cite{pca}. In \cite{nonsym}, a smoothed aggregation method is described for nonsymmetric linear systems that arise from partial differential equation (PDE) discretization, and which requires approximate calculation of the minimal singular triplet of the problem matrix in a setup phase of the solver. Similarly, calculating dominant or minimal eigenpairs of SPD matrices also has many applications, see, e.g., 
\cite{lobpcg,borzi,hetmaniuk,kushnir,mgeigen}.

The computation of a few extremal singular triplets of large sparse matrices has been the focus of many research efforts, see, for example, \cite{kokio,baglama} and the numerous references therein. In recent times, Lanczos bidiagonalization methods and subspace iteration methods have received significant attention. 
Singular triplets can also be computed by applying symmetric eigenvalue solvers to $A^t\,A$ or the augmented operator
\begin{align}
\label{eq:augment}
X=\left[
\begin{array}{cc}
0 & A\\
A^t & 0
\end{array}
\right],
\end{align}
but the first approach can lead to poor accuracy of the computed singular values when $A$ is ill-conditioned. For the second approach the amount of storage and work required can be prohibitive, the number of iterations required to compute a given number of singular values increases, and the indefiniteness of operator $X$ has to be dealt with \cite{larsen}. For these reasons, methods are being pursued that avoid working on operators $A^t\,A$ and $X$ \cite{larsen,kokio,baglama}, and we do the same in this paper.
It appears that multilevel methods have not been explored yet for the calculation of singular triplets working directly on \eqs{f}. This is, perhaps, not surprising, since AMG methods for the SPD eigenproblem are also still quite a young area \cite{bamg,eis,borzi,hetmaniuk,kushnir}. It can be expected that AMG methods for extremal singular triplets will be competitive for problems in which the extremal singular values are highly clustered and the extremal singular vectors are similar to each other such that they can be represented well collectively by an interpolation operator that interpolates coarse-grid representations of the singular vectors to the fine grid. Nonsymmetric discretized elliptic PDE operators are expected to have this kind of spectral decomposition. We will investigate such a problem in the numerical results section of our paper, but we think that it is also interesting to investigate the applicability and performance of our algorithm for other, more general SVD problems, and we do so in the numerical results section as well. Numerical results will also be presented for SPD eigenproblems, since our algorithm offers a new extension of previous approaches for this type of problems as well.

Algebraic multigrid was originally developed for solving sparse systems of linear equations (see \cite{amg1} and references in \cite{stueben} and \cite{mgtut}). Over the years, its applicability has been extended in several ways, including to SPD eigenvalue problems \cite{bamg,eis,borzi,hetmaniuk,kushnir}. 
The AMG method we propose belongs to the class of {\em self-learning} AMG methods (we borrow this term from \cite{ibamg}). In these methods, a multigrid hierarchy is built with interpolation operators that are determined adaptively and iteratively over several multilevel cycles, to match approximately the vectors that are of interest in the problem at hand. For linear system solvers, these are the vectors that lie close to the null-space of the matrix, and for eigenvalue problems, they are the desired eigenvectors. In our new method for singular triplets, they will be the desired singular vectors.
Self-learning AMG solvers are an active area of research and have been developed for solving linear equation systems, SPD eigenproblems, and Markov chain problems, see, for example,
\cite{amg1,bamg,asa,aamg,eis,sam,mcamg,bamgMarkov,fly,over,ira,ibamg,nonsym,bamgQcd}.
Our AMG method is also {\em collective}, in that it strives to represent several singular vectors by a single interpolation matrix for efficiency.

The AMG method we propose for computing dominant singular triplets consists of two multilevel phases.
It combines and extends two existing AMG approaches for the SPD eigenproblem, that were proposed by Borzi and Borzi in \cite{borzi} and by Kushnir, Galun and Brandt in \cite{kushnir}.
In the first, multiplicative phase (setup phase), we calculate tentative singular triplets and a multigrid hierarchy with interpolation operators that approximately fit the tentative singular vectors in a collective and self-learning manner. This phase uses power method relaxation and multiplicative coarse-grid update formulas for the tentative singular vectors. We use the bootstrap framework \cite{bamg} in this phase with least-squares fitting and random initial singular vectors, in  a way similar to the approach described by Kushnir, Galun and Brandt in \cite{kushnir} for calculating minimal eigenpairs of an SPD matrix. In other related work, the setup phase of the algorithm described in \cite{nonsym} calculates an approximation of the singular vectors that correspond to the smallest singular value of a square nonsymmetric matrix, in a way that is less general than but similar to our multiplicative phase.
In \cite{kushnir}, great care is taken to try to make the interpolation operators highly accurate for all eigenvectors, in the spirit of the exact interpolation scheme (EIS) \cite{eis}, leading to an eigenvalue solver that only employs this first, multiplicative phase, with accuracy limited to the accuracy by which the single interpolation operator represents each eigenvector. In our approach, however, we use generic interpolation that fits the tentative singular vectors only approximately, and we employ a second, additive phase (solve phase), in which the tentative singular triplets are improved up to the desired accuracy by using an additive correction scheme with fixed interpolation operators, combined with a Ritz update. Our additive phase is similar to the approach described by Borzi and Borzi in \cite{borzi} for calculating minimal eigenpairs of an SPD matrix (which itself is an extension of \cite{mgeigen}), but in \cite{borzi} standard AMG interpolation is used, and there is no initial multiplicative self-learning phase. Our hybrid multiplicative-additive approach results in a new AMG method for extremal singular triplets that combines two desirable properties: it allows for high-accuracy convergence when desired, and it is flexible enough to deal efficiently with a variety of problems due to its self-learning properties. The specialization of our algorithm to the SPD eigenpair case also leads to a new extension of the AMG eigenvalue algorithms of \cite{borzi} and \cite{kushnir} that has the same desirable properties.

The remainder of this paper is structured as follows. In the next section we give a brief description of multiplicative and additive two-level schemes for solving $(A - \lambda I) \, x = 0$, with $A$ a square SPD matrix and $\lambda$ an assumed given, fixed eigenvalue. This will serve to elucidate under which circumstances multiplicative and additive update formulations can be equivalent for calculating eigenvectors, and to illustrate when it may be beneficial for accuracy and computational cost reasons to append a phase with additive cycles to an initial multiplicative, self-learning phase. This will set the stage for the description of the multiplicative (setup) phase of our singular triplet algorithm in Section \ref{sec:mult}. This section also introduces a suitable generalization of the SVD for formulating the coarse-level problems. Section \ref{sec:add} then describes the additive (solve) phase of the algorithm. Section \ref{sec:spec} describes how it can be extended and specialized to the case of square matrices, minimal singular triplets and extremal eigenpairs of SPD matrices. Section \ref{sec:numres} contains extensive numerical evaluation of our algorithm, 
and Section \ref{sec:conc} concludes.

\section{Two-level Methods for $(A-\lambda\,I)\,x=0$}
In this section, we consider multiplicative and additive two-level methods for calculating an eigenvector of a square SPD matrix $A \in I\!\!R^{m \times m}$, assuming, for now, that the eigenvalue $\lambda$ is known. This academic discussion serves to highlight the principles behind the multiplicative and additive approaches, and how they are related and can be combined for calculating eigenvectors. The insights gained will motivate the approach of our multilevel algorithm for calculating dominant singular triplets. Assuming eigenvalue $\lambda$ is known, we seek a nontrivial solution $x$ to equation
\begin{align}
(A -\lambda \, I)\, x = 0,
\end{align}
with $I$ generically denoting the unit matrix.
For definiteness, we can simply assume that $\dim (\ker(A-\lambda\,I))=1$, and that we seek a solution with $\|x\|_2=1$. We will consider two-level iterative schemes with relaxations on the fine level (or fine grid) combined with coarse-grid corrections that are obtained via solving a smaller problem on a coarse grid.
\subsection{Multiplicative Correction Scheme}
Let $x^{(i)}$ be the current fine-grid approximation and $e_{mult}^{(i)}$ be its multiplicative error, such that
\begin{align}
x = \dia{x^{(i)}} \, e_{mult}^{(i)},
\label{eq:multError}
\end{align}
where $x$ is the exact solution of the problem, and $\dia{x^{(i)}}$ is a diagonal matrix with $x^{(i)}$ on its diagonal.
Note that, at convergence, when $x^{(i)}=x$, the multiplicative error satisfies $e_{mult}^{(i)}=\one$, with $\one$ the vector of all ones.
The problem at hand can be rewritten as
\begin{align}
(A -\lambda \, I) \, \dia{x^{(i)}} \, e_{mult}^{(i)} =0,
\label{eq:a}
\end{align}
in which we seek the unknown multiplicative error $e_{mult}^{(i)}$.
We consider a coarse grid with $m_c$ unknowns (which may be a subset of the fine-grid unknowns), and, instead of the fine-level multiplicative error, $e_{mult}^{(i)}$, we seek to compute a coarse-grid multiplicative error $e_{mult,c}$, which, when interpolated up to the fine grid, would approximately equal the unknown fine-level multiplicative error. This may be an inexpensive way to improve the fine-level error, since $e_{mult,c}$ can be computed inexpensively on the coarse grid.
So we seek $e_{mult,c}$ such that
\begin{align}
Q \, e_{mult,c} \approx e_{mult}^{(i)},
\label{eq:b}
\end{align}
with $Q \in I\!\!R^{m \times m_c}$ a coarse-to-fine interpolation matrix for the error, which we require to satisfy
$Q \, \one_c=\one$ (with $\one_c$ the coarse-level vector of all ones).
Combining Eqs.\ (\ref{eq:a}) and (\ref{eq:b}) and with the help of a restriction operator, $R \in I\!\!R^{m_c \times m}$, we arrive at the following $m_c \times m_c$ system of equations for $e_{mult,c}$:
\begin{align}
R\, (A -\lambda \, I) \, \dia{x^{(i)}} \, Q \, e_{mult,c} =0.
\end{align}
\eqs{multError}  and (\ref{eq:b}) then lead to the multiplicative coarse-grid correction formula for the fine-grid approximation:
\begin{align}
x^{(i+1)} = \dia{x^{(i)}} \, Q \, e_{mult,c}.
\label{eq:cgcMult}
\end{align}
It is also useful to define the interpolation matrix $P \in I\!\!R^{m \times m_c}$, given by
\begin{align}
P= \dia{x^{(i)}} \, Q ,
\label{eq:Pmult}
\end{align}
which has the property that the current fine-grid approximation, $x^{(i)}$, lies exactly in its range, namely,
\begin{align}
x^{(i)}=P\,\one_c. 
\label{eq:range1}
\end{align}
More generally, we have that there exists a coarse-level vector $e_{mult,c}^{(i)}$ (and we know which one) such that
\begin{align}
x^{(i)} = P \, e_{mult,c}^{(i)}.
\label{eq:range}
\end{align}
Using interpolation operator $P$, the coarse-grid equation becomes
\begin{align}
R \, (A-\lambda \, I) \, P \, e_{mult,c} =0,
\end{align}
and the multiplicative coarse-grid correction formula
\begin{align}
x^{(i+1)} = P \, e_{mult,c}.
\label{eq:multCorr}
\end{align}
Considering a two-level method with coarse-grid correction according to \eq{multCorr}, we note that, for the exact solution $x$ to be a fixed point of such a two-level method, $x$ indeed needs to lie exactly in the range of $P$ at convergence, and the multiplicative correction scheme described above assures this by having $x^{(i)}$ lie exactly in the range of $P$ in each iteration. It is important to realize that this is required for the multiplicative scheme to converge to the exact solution. Since $P$ (and possibly $R$) change in every iteration to adapt to the solution sought, we call this multiplicative two-level scheme self-learning.

Note also, that one can always consider a rescaled coarse-level unknown quantity, say $x_c$,
using a diagonal scaling matrix $W$,
\begin{align}
e_{mult,c} = W\, x_c,
\end{align}
and formulate the multiplicative scheme in terms of $x_c$. For example, defining $\hat{P}=P\,W$, one solves coarse-level problem $R\,A\,\hat{P}\,x_c=0$ and corrects with multiplicative coarse-grid correction formula $x^{(i+1)}=\hat{P}\,x_c$. Such an $x_c$ is generally not a multiplicative error anymore (since it does not hold that $x_c=\one_c$ at convergence), but can be some kind of coarse-level representation of the fine-level exact solution $x$. The current iterate, $x^{(i)}$, still lies exactly in the range of $\hat{P}$ in each iteration. This viewpoint is adopted in the derivation of the multiplicative correction scheme as an Exact Interpolation Scheme \cite{eis}, while our derivation is the more common viewpoint in the context of Markov chains \cite{horton,sam,mcamg}. We take this viewpoint here to highlight the similarity between the multiplicative and additive error correction formalisms, as will be discussed now.
\subsection{Additive Correction Scheme}
Multigrid for linear systems of equations is normally formulated in an additive-correction framework \cite{mgtut}.
Define residual $r^{(i)}$ of current approximation $x^{(i)}$ as
\begin{align}
r^{(i)} = -(A - \lambda \,I) \,x^{(i)}.
\label{eq:resid}
\end{align}
The additive error, $e_{add}^{(i)}$, satisfies
\begin{align}
x = e_{add}^{(i)} +x^{(i)},
\label{eq:addError}
\end{align}
and the problem at hand can be rewritten as error equation
\begin{align}
(A - \lambda \,I) \, e_{add}^{(i)}= r^{(i)}.
\label{eq:c}
\end{align}
We seek to compute a coarse-grid additive error, $e_{add,c}$, which, when interpolated up to the fine grid, would approximately equal the unknown fine-level additive error. So we seek $e_{add,c}$ such that
\begin{align}
P \, e_{add,c} \approx e_{add}^{(i)},
\label{eq:d}
\end{align}
for some coarse-to-fine interpolation operator $P \in I\!\!R^{m \times m_c}$.
Combining \eqs{c} and (\ref{eq:d}) and
with the help of a restriction operator, $R \in I\!\!R^{m_c \times m}$, we arrive at the following $m_c \times m_c$ system of equations for $e_{add,c}$:
\begin{align}
R \, (A-\lambda\, I) \, P \, e_{add,c} = R\,r^{(i)}.
\label{eq:cgeAdd}
\end{align}
\eqs{addError} and (\ref{eq:d}) then lead to the additive coarse-grid correction formula for the fine-grid approximation:
\begin{align}
x^{(i+1)}=x^{(i)}+P\, e_{add,c}.
\end{align}
Fast convergence of the two-level process requires that additive error components that are not significantly reduced by fine-level relaxation lie approximately in the range of $P$ (and can thus be removed by coarse-grid correction).
First consider $P=\dia{x^{(i)}}\, Q$ as above, see \eq{Pmult}. This means that $x^{(i)}$ lies exactly in the range of $P$,
and, close to convergence, $x$ will lie approximately in the range of $P$ as well. This means that $e_{add}^{(i)}=x-x^{(i)}$ also lies approximately in the range of $P$, so the `self-learning' $P$ from \eq{Pmult} is expected to give a suitable interpolation operator also for the additive correction scheme.

In fact, if the same $P$ is used as in the multiplicative method (with
$x^{(i)}$ lying exactly in the range of $P$, \eq{range}) in every iteration of the additive scheme, and if $R$ is also taken the same as in the multiplicative scheme, then the additive and multiplicative schemes are exactly equivalent. This can easily be seen as follows. First, using \eq{resid} and \eq{range1}, additive coarse-level equation \eq{cgeAdd} can be rewritten as
\begin{alignat}{1}
&R\,(A-\lambda\, I)\,P\, e_{add,c}=R\,r^{(i)}=-R\,(A-\lambda\, I)\,x^{(i)}=-R\,(A-\lambda\, I)\,P\,\one_c,\\
&R\,(A-\lambda\, I)\,P\,(e_{add,c}+\one_c)=0,
\end{alignat}
and by identifying
\begin{align}
e_{add,c}+\one_c=e_{mult,c},
\label{eq:e}
\end{align}
one obtains the multiplicative coarse-grid equation,
\begin{align}
R\,A\,P\,e_{mult,c}=0.
\end{align}
\eq{e} has the nice interpretation that at convergence, on the coarse grid, the additive error, $e_{add,c}$, vanishes, and the multiplicative error, $e_{mult,c}$, equals $\one_c$.
Similarly, using \eq{range}, the multiplicative coarse-grid correction formula can be obtained from the additive coarse-grid correction formula:
\begin{align}
&x^{(i+1)}=x^{(i)}+P\, e_{add,c}=P\,(\one_c+e_{add,c}), \nonumber\\
&x^{(i+1)}=P\,e_{mult,c}.
\end{align}
This shows that the additive scheme with $R$ and $P$ chosen (in every iteration) as in the multiplicative scheme (with current iterate $x^{(i)}$ exactly in the range of $P$), is fully equivalent with the multiplicative scheme (in exact arithmetic). 

However, unlike the multiplicative scheme, the additive scheme can still converge if $x^{(i)}$ (and thus $e_{add}^{(i)}$) lies only approximately in the range of $P$. Therefore, one approach to obtaining a convergent additive method is to first (adaptively) determine $P$ (and $R$) in a few multiplicative cycles, and then freeze $R$ and $P$ for subsequent additive cycles. Additive cycles with frozen $R$ and $P$ are much cheaper computationally than cycles in which $P$ (and possibly $R$) are modified in each iteration, often without sacrificing convergence speed too much, and the resulting hybrid method may be significantly cheaper than a fully adaptive method, since all coarse-level operators are kept constant in the additive phase. This is one reason to consider hybrid multiplicative-additive methods for eigenvalues (see also \cite{bamgMarkov,fly,over} for application of this approach in the Markov chain context). The original multigrid method for solving linear equation systems is formulated in the additive framework, with fixed $R$ and $P$ that are determined using a-priori knowledge of the problem. Self-learning methods for linear systems of equations first perform some multiplicative, self-learning setup cycles to determine suitable interpolation operators, before proceeding with additive cycles with fixed interpolation \cite{amg1,bamg,asa,aamg,eis,ibamg,nonsym,bamgQcd}. In the context of self-learning solvers for linear equation systems, the first, multiplicative self-learning phase is often called the setup phase, because it is merely used for setting up the solver, and no approximation to the solution of the linear system $A\,x=b$ is sought in the multiplicative phase. When computing eigenpairs, however, the multiplicative phase is not merely a setup phase, since it also iterates on approximations for the eigenpairs, and it can be used as an eigensolver by itself, as in \cite{horton,eis,sam,mcamg,kushnir}. For this reason, we more generally refer to it as a multiplicative phase, in the present context of eigenvalue and singular value problems.

In this paper we consider methods to compute a few dominant singular triplets or eigenpairs that are hybrid multiplicative-additive, not only for performance reasons, but mainly for the following reason: we will seek to formulate multilevel methods in which, for efficiency, the same interpolation matrix $P$ can be used to approximate several singular vectors or eigenvectors at the same time. This interpolation matrix will not contain all of these singular vectors or eigenvectors in its range exactly, so a multiplicative scheme will only converge up to the accuracy by which the vectors are collectively represented by the interpolation matrix. A multiplicative scheme will be used to initiate the calculations and approximately identify the dominant singular vectors or eigenvectors, determining suitable interpolation operators in the process. Rather than attempting to construct interpolation operators that are very accurate for all vectors sought (as is done in \cite{kushnir} for eigenvalue calculation), we will switch to an additive scheme in our singular triplet method, mainly because it can converge with high accuracy for all vectors sought, and with the added benefit that it will be inexpensive per cycle.

\section{AMG SVD Algorithm: Multiplicative Phase}
\label{sec:mult}
We now go back to the general setting of our paper in which we want to compute dominant singular triplets $(\sigma,u,v)$ of rectangular matrix $A \in I\!\!R^{m \times n}$, satisfying
\begin{alignat}{1}
A \, v &= \sigma \, u, \nonumber\\
 A^t \, u &= \sigma \, v.
\label{eq:svdFine}
\end{alignat}
In this section, we formulate the multiplicative phase of the algorithm. (Note that we will redefine the interpolation and restriction matrices $P, \, Q, \, R,$ etc.)

\subsection{Coarse-level Equations}
Consider interpolation matrices $P$ for $u$ and $Q$ for $v$, with $P \in I\!\!R^{m \times m_c}$
and $Q \in I\!\!R^{n \times n_c}$, and $P$ and $Q$ of full rank.
First assume that $u$ lies exactly in the range of $P$, and $v$ in the range of $Q$, so
\begin{alignat}{1}
u &= P \, u_c, \nonumber \\
v &= Q \, v_c,
\label{eq:ranges}
\end{alignat}
for some coarse-level vectors $u_c$ and $v_c$.
We define coarse-level equations
\begin{alignat}{1}
P^t \, A\, Q \, v_c &=\sigma \, P^t \, B \, P \, u_c,\nonumber \\
Q^t \, A^t\, P \, u_c &=\sigma \, Q^t \, C \, Q \, v_c,
\end{alignat}
and coarse-level operators
\begin{alignat}{1}
\label{eq:coarseOp}
A_c &= P^t \, A\, Q,\nonumber \\
B_c &= P^t \, B \, P,\\
C_c &= Q^t \, C \, Q, \nonumber
\end{alignat}
with, for the finest-level operators, $B=I_m$ and $C=I_n$.
The coarse-level version of fine-level equations (\ref{eq:svdFine}) is then given by
\begin{alignat}{1}
A_c \, v_c &=\sigma \, B_c \, u_c, \nonumber\\
A_c^t\, u_c &=\sigma \, C_c \, v_c.
\label{eq:svdCoarse}
\end{alignat}
The intuition behind this approach is as follows: the coarse-level equations can be expected to be useful for finding triplet $(\sigma,u,v)$, since, if $(\sigma,u,v)$ is a singular triplet of $A$ and \eqs{ranges} are assumed, then $(\sigma,u_c,v_c)$ is a singular triplet of $A_c$. So one can see that, if $P$ and $Q$ can be constructed such that $u$ and $v$ lie exactly in their respective ranges (\eqs{ranges}), then a coarse-level solve can give us $(\sigma,u,v)$ exactly. The same reasoning applies when coarsening is repeated recursively. Note that the $B_c$ and $C_c$ on all recursive levels are symmetric positive definite (SPD) since the $P$ and $Q$ are chosen of full rank. We will now consider methods to build $P$ and $Q$ such that $u$ and $v$ lie in their respective ranges approximately.

\subsection{Generalization of Singular Value Problem}
Coarse-level equations (\ref{eq:svdCoarse}) are of the form
\begin{alignat}{1}
A \, v &=\sigma \, B \, u, \nonumber\\
A^t\, u &=\sigma \, C \, v,
\label{eq:svdGenSeparate}
\end{alignat}
with $B$ and $C$ SPD.
The coarse-level equations motivate the following generalization of the singular value decomposition.
\begin{definition}[Generalized singular value decomposition]
The generalized singular value decomposition of $A \in I\!\!R^{m \times n}$ with respect to $B\in I\!\!R^{m \times m}$ and $C\in I\!\!R^{n \times n}$, with $B$ and $C$ SPD, is given by
\begin{align}
A = B \, U \, \Sigma \, V^t \, C, 
\label{eq:genSVDDef}
\end{align}
with $U \in I\!\!R^{m \times m}$, $V \in I\!\!R^{n \times n}$ and $\Sigma \in I\!\!R^{m \times n}$. The columns of $U$ are called the left generalized singular vectors, and the columns of $V$ are called the right generalized singular vectors. They satisfy the orthogonality relations $U^t\, B \, U=I_m=U\, B \, U^t$ and $V^t\, C\, V=I_n=V\, C\, V^t$.
Matrix $\Sigma$ has the $l=\min(m,n)$ real nonnegative generalized singular values $\sigma_1 \ge \sigma_2 \ge \ldots \ge \sigma_l \ge 0$ on its diagonal.
\eqs{svdGenSeparate} are called the generalized singular value problem for matrix $A$ with respect to matrices $B$ and $C$.
\label{def:genSVD}
\end{definition}

It is easy to see that the generalized singular triplets $(\sigma,u,v)$ of generalized SVD \noeq{genSVDDef} satisfy
\eqs{svdGenSeparate}. When $B=I_m$ and $C=I_n$, generalized SVD \noeq{genSVDDef} reduces to the standard SVD.

It has to be remarked that the notion of generalized SVD as defined above is different from the more commonly used generalized SVD of $A  \in I\!\!R^{m \times n}$ with respect to $B \in I\!\!R^{p \times n}$ (with $m\ge n$), as, for example, defined in \cite{golub}, p.~471. \defi{genSVD} is the sense of generalized SVD that we need in this paper. While \eq{genSVDDef} is a natural generalization of the singular value decomposition and relates to it in the same way the generalized eigenvalue problem (as it is commonly defined) relates to the standard eigenvalue problem, we have not been able to find it in the literature yet. In what follows, we formulate the properties of the generalized SVD that are useful for the calculations to be done in our multilevel cycles. We discuss existence and uniqueness, which is important for the well-posedness of our multilevel cycles, and we explain how the generalized SVD can be calculated, which we will need to do on the coarsest level of our multilevel cycles.
\begin{theorem}
Generalized SVD \noeq{genSVDDef} has the same existence and uniqueness properties as the standard SVD.
\end{theorem}
\begin{proof}
This follows from a simple change of variables: with
\begin{alignat}{1}
T  &= B^{1/2}\,U, \nonumber \\
W &= C^{1/2}\,V, \\
D &= B^{-1/2} \, A \, C^{-1/2}, \nonumber
\label{eq:svdTrafo}
\end{alignat}
generalized SVD \noeq{genSVDDef} can be rewritten as a standard SVD
\begin{align}
D = T \, \Sigma \, W^t. 
\end{align}
\end{proof}

This change of variables provides a first manner of computing generalized SVD \noeq{genSVDDef} using standard SVD algorithms. An alternative way of computing generalized SVD \noeq{genSVDDef} proceeds as follows.
Let
\begin{align}
\label{eq:X}
X&=
\left[
\begin{array}{cc}
0 & A\\
A^t & 0
\end{array}
\right],\\
Y&=
\left[
\begin{array}{cc}
B & 0\\
0 & C
\end{array}
\right].
\end{align}
It is clear that $X$ is symmetric and $Y$ is SPD, and 
\begin{align}
\left(
X-\sigma \,
Y
\right) \,
z
=0,
\label{eq:genEig}
\end{align}
is a symmetric generalized eigenvalue problem of size $(m+n)\times(m+n)$, with $m+n$ real eigenvalues $\sigma_j$ and associated eigenvectors $[u_j^t \ v_j^t]^t$, which can be chosen orthonormal with respect to $Y$.
The following theorem indicates how the solutions of this generalized eigenvalue problem can be used to compute the generalized singular triplets of generalized SVD \noeq{genSVDDef}.
\begin{theorem}
Let $A \in I\!\!R^{m \times n}$, $B \in I\!\!R^{m \times m}$ and $C \in I\!\!R^{n \times n}$, with $B$ and $C$ SPD.
Let $l=\min(m,n)$. Then generalized eigenvalue problem
\begin{align}
\left(
\left[
\begin{array}{cc}
0 & A\\
A^t & 0
\end{array}
\right]
-\sigma \,
\left[
\begin{array}{cc}
B & 0\\
0 & C
\end{array}
\right]
\right) \,
\left[
\begin{array}{c}
u\\
v
\end{array}
\right]
=0,
\label{eq:genSVD}
\end{align}
has $m+n$ solution triplets $(\sigma,u,v)$ with linearly independent eigenvectors $[u^t \, v^t]^t\ne0$.
There are $l$ independent solutions with $\sigma_j\ge0$ and vectors $u_j$ and $v_j$ satisfying orthogonality relations $u_j^t\,B\,u_i=\delta_{i,j}$ and $v_j^t\,C\,v_i=\delta_{i,j}$ ($j=1,\ldots,l$). The triplets $(\sigma_j,u_j,v_j)$ are the generalized singular triplets of $A$ with respect to $B$ and $C$.
Furthermore, there are $l$ independent solutions $(-\sigma_j,u_j,-v_j)$.
Finally, there are ${\rm abs}(m-n)=m+n-2\,l$ independent solutions with $\sigma=0$ and either $u=0$ or $v=0$.
\label{thm:genSVD}
\end{theorem}
\begin{proof}
This follows directly from the variable transformations \eqs{svdTrafo}, which transform generalized eigenvalue problem
\eq{genSVD} into eigenvalue problem
\begin{align}
\left(
\left[
\begin{array}{cc}
0 & D\\
D^t & 0
\end{array}
\right]
-\sigma \,
\left[
\begin{array}{cc}
I_m & 0\\
0 & I_n
\end{array}
\right]
\right) \,
\left[
\begin{array}{c}
t\\
w
\end{array}
\right]
=0,
\label{eq:genSVDTraf2}
\end{align}
which has the properties listed in the theorem, see, for example, \cite{golub}, p.~427.
\end{proof}

A third possible way to calculate generalized SVD \noeq{genSVDDef} is by solving for the left and right generalized singular vectors separately, using
\begin{alignat}{1}
(A^t \, B^{-1} \, A) \, v = \sigma^2 \, C \, v, \nonumber\\
(A \, C^{-1} \, A^t) \, u = \sigma^2 \, B \, u.
\label{eq:svdGenSeparate2}
\end{alignat}

\subsection{Bootstrap AMG V-cycles}
In this section, we describe how we use the bootstrap AMG approach \cite{bamg} to find approximations of the desired $n_b$ dominant singular vectors and values, and adaptively determine interpolation operators that approximately fit the singular vectors. We follow the approach described in \cite{kushnir}. For completeness and definiteness, we briefly describe all steps in the process, with some details filled in in subsequent sections.

We first describe the initial BAMG V-cycle. We start out on the finest level by choosing $n_t$ random test vectors for each of $u$ and $v$, and we place them in the columns of $U_t$ and $V_t$, respectively. We relax on the test vectors (using a few iterations of the SVD power method for \eq{svdFine}, see below) such that components with small $\sigma$ are damped and components with large $\sigma$ become dominant in the test vectors. We coarsen the finest grid (see below) and determine interpolation operators $P$, $Q$, where $P$ fits the vectors in $U_t$ (in a least-squares sense), such that they lie approximately in the range of $P$, and $Q$ fits the vectors in $V_t$, such that they lie approximately in the range of $Q$. We also build coarse-level operators $A_c$, $B_c$, and $C_c$ according to \eqs{coarseOp}.
We then restrict the fine-level $U_t$ and $V_t$ (by injection) to the first coarse level, and obtain coarse versions of the test vectors, stored in the columns of $U_{c,t}$ and $V_{c,t}$. We relax on $U_{c,t}$ and $V_{c,t}$ with the power method applied to \eqs{svdCoarse}.
The whole process of building new, coarser interpolation operators $P$ and $Q$ and operators $A_c$, $B_c$, $C_c$, by restricting $U_{c,t}$ and $V_{c,t}$ is then repeated recursively, up to some coarse level where the problem is small enough for a direct generalized SVD calculation.

On the coarsest level, we compute $n_b$ dominant singular triplets by a direct decomposition, and store them in vector
$\sigma_b$ and matrices $U_b$ and $V_b$. These singular triplets are the starting approximations for our desired dominant singular triplets, and will be improved in this cycle and subsequent cycles. We call the singular vectors of these triplets the boot (singular) vectors, and use the subscript $b$ to refer to them. (We distinguish these from the initially random test vectors in $U_t$ and $V_t$, which are used to get the process going and sustain it, but do not directly lead to the desired $n_b$ singular triplets themselves.) Note that we denote by $\sigma_b$ a vector with $n_b$ components that holds approximations for the dominant singular values sought. 

In the upward phase of the first BAMG V-cycle, starting from the coarsest level, we recursively interpolate the boot singular vectors $U_b$ and $V_b$ up to the next finer level, using the interpolation operators $P$ and $Q$ of the current level, according to multiplicative update formulas \eqs{ranges}. On each finer level, we first relax on the boot vectors using \eqs{svdCoarse} with the singular values in $\sigma_b$ fixed, and then update the elements of $\sigma_b$ by recalculating the Rayleigh quotient for each pair of boot vectors (see below). Note that the test vectors $U_t$ and $V_t$ are not used in the upward phase of the V-cycle.

This initial BAMG multiplicative V-cycle can be followed by several additional multiplicative V-cycles.
In the downward sweep of each of these additional cycles, one relaxes $U_t$ and $V_t$ as in the first V-cycle. In addition, one also relaxes the $U_b$ and $V_b$, and improves the $\sigma_b$ on each level, as in the upward sweep of the first cycle. At each level, the vectors in both $U_t$ and $U_b$ are used to fit $P$, and the vectors in both $V_t$ and $V_b$ to fit $Q$. Then $A_c$, $B_c$ and $C_c$ are also rebuilt using the new $P$ and $Q$.
The upward sweeps of the additional multiplicative cycles are the same as in the initial multiplicative cycle.
At the end of every V-cycle, we optionally also apply a collective Ritz projection step (see below) to improve the boot vectors $U_b$, $V_b$ and singular value approximations $\sigma_b$. We do so for the numerical tests reported in \sect{numres}.

Note that in this paper we use only the simplest type of multilevel cycles, namely, V-cycles. More sophisticated cycles including W-cycles and full multigrid (FMG) cycles \cite{mgtut,borzi,hetmaniuk,kushnir} can be considered and may lead to improved results, but for simplicity we only use V-cycles here. In the following sections we will give the details of the relaxation schemes, coarsest-level solve, coarsening and interpolation used in our BAMG cycles.

\subsection{Relaxation Scheme for the Test Vectors}
Seeking dominant singular triplets, we base relaxation for the initially random test vectors on the power method applied to \eq{svdGenSeparate}.
On any level, given an initial $u_j$, we solve for $v_j$ from
\begin{alignat}{1}
&A^t\, u_j =C \, \bar{v}_j, \nonumber\\
&v_j=\bar{v}_j/  (\bar{v}_j^t C \bar{v}_j)^{1/2},
\label{eq:power1}
\end{alignat}
and then for $u_j$ from
\begin{alignat}{1}
&A\, v_j =B \, \bar{u}_j, \nonumber\\
&u_j=\bar{u}_j/ (\bar{u}_j^t B \bar{u}_j)^{1/2}.
\end{alignat}
This can be repeated $\mu_t$ times on each level.
In practice, we solve for the new $\bar{v}_j$ and $\bar{u}_j$ in an inexact way, by performing $\mu_{t,J}$ inner iteration steps of weighted Jacobi with weight $\omega_J$. For example, for \eq{power1} we iterate on:
\begin{align}
\bar{v}_j^{(i+1)}=\bar{v}_j^{(i)}-\omega_J \, D_C^{-1} \, (C \, \bar{v}_j^{(i)} - A^t\, u_j)
\end{align}
with $\bar{v}_j^{(0)}=v_j$ initially and with the iteration index of the weighted Jacobi procedure indicated in superscript.
Here, $D_C$ is a diagonal matrix with the diagonal of the SPD matrix $C$ on its diagonal. In the numerical results reported in \sect{numres}, we use $\omega_J=0.7$ and $\mu_{t,J}=1$.
\subsection{Relaxation Scheme for the Boot Vectors and Update Formulas for the Singular Values}
\label{subsec:relaxBoot}
For the boot vectors, we relax on
\begin{alignat}{1}
A \, v &=\sigma \, B \, u + \kappa,
\label{eq:genSVDresid1} \\
A^t\, u &=\sigma \, C \, v + \tau.
\label{eq:genSVDresid2}
\end{alignat}
(Note that in the multiplicative phase $\kappa=0$ and $\tau=0$ on all levels, but the additive phase will require nonvanishing $\kappa$ and $\tau$, so we already include them in the formulation here.)
On any level, given an initial $\sigma_j$, $u_j$ and $v_j$, we solve for a new $u_j$ from \eq{genSVDresid1}, and then for a new $v_j$ from \eq{genSVDresid2}. This amounts to a block Gauss-Seidel (GS) scheme for equation system \ (\ref{eq:genSVDresid1})-(\ref{eq:genSVDresid2}). For dominant $\sigma$s, Eqs.\ (\ref{eq:genSVDresid1})-(\ref{eq:genSVDresid2})
\begin{align}
\left(
X-\sigma \,
Y
\right) \,
[u^t \ v^t]^t
=[\kappa^t \, \tau^t]^t,
\label{eq:big}
\end{align}
may be close to diagonally dominant, so this will work well in many cases. For some problems or on coarser levels, the block GS approach may not converge well, and Kaczmarz relaxation \cite{tanabe,kushnir} (see also below) on \eq{big} or its blocks may be preferable.
In our block GS approach, we again approximate the solutions of Eqs.\ (\ref{eq:genSVDresid1})-(\ref{eq:genSVDresid2}) in an inexact way, by performing $\mu_{b,J}$ inner iteration steps of weighted Jacobi.
For example, for \eq{genSVDresid1} we iterate on:
\begin{align}
u_j^{(i+1)}=u_j^{(i)}-\omega_J \, D_B^{-1} \, (B \, u_j^{(i)} - (A\, v_j-\kappa)/\sigma_j).
\end{align}
In the numerical results reported in \sect{numres}, we use $\mu_{b,J}=1$.
In the multiplicative phase, with $\kappa=0$, $\tau=0$ on all levels, we also update the $\sigma$s after every outer relaxation iteration on each level. The easiest way to do this is to use Rayleigh quotient formula
\begin{align}
\sigma=\frac{u^t A v}{(u^t B u)^{1/2} \, (v^t C v)^{1/2}}
\label{eq:Rayleigh}
\end{align}
for each boot singular triplet, which is what we do in the numerical results presented in \sect{numres}.
\subsection{Coarsest-grid Solution}
Each time the coarsest level is reached, we determine new approximations for the $n_b$ coarsest-level boot triplets by direct computation of the coarsest-level generalized SVD, \eq{genSVDDef}. The $n_b$ singular triplets with the largest singular values are selected as the new boot singular triplets. In our implementation, we choose to solve generalized eigenproblem \noeq{genSVD} of \thm{genSVD} using a direct eigendecomposition algorithm.
\subsection{Building $P$ and $Q$: Coarsening and Sparsity Patterns}
In order to build interpolation operators $P$ and $Q$, at each level, we first coarsen the sets of unknowns in $u \in I\!\!R^{m}$ and $v \in I\!\!R^{n}$ by choosing a set of $m_c$ coarse-grid variables, $C_u$, out of the $m$ fine-level variables for $u$, and by choosing a set of $n_c$ coarse-grid variables, $C_v$, out of the $n$ fine-level variables for $v$. The coarse variables are called coarse-grid points or C-points. The fine-level $u$-variables that are not selected as C-points are called F-points and are denoted by the set $F_u$. Similarly, the F-points of the fine-level $v$-variables are denoted by $F_v$. Well-known algorithms from AMG are used to determine $C_u$ and $C_v$ and the sparsity patterns of $P$ and $Q$ on each level, based on the idea of strength of connection in the operator matrices $A$.
After this coarsening process, the matrix elements of $P$ and $Q$ are determined using a least-squares approach in such a way that the test vectors $U_t$ and $V_t$ (and, after the initial cycle, also the boot vectors $U_b$ and $V_b$) lie approximately in the ranges of $P$ and $Q$, respectively.

For the coarsening process for the $u$-variables, we propose to apply standard AMG coarsening methods to matrix $A A^t$, and we base coarsening of the $v$-variables on $A^t A$. (If $A$ is square or square and symmetric, other choices can be made, see below.)

\begin{algorithm}[H]
\caption{one-pass Ruge-Stueben coarsening algorithm}
set $U$ $\leftarrow$ all fine-level points; $C$ $\leftarrow$ empty; $F$ $\leftarrow$ empty\;
for all fine-level points $i$, set $\lambda_i$ $\leftarrow$ number of points strongly influenced by $i$\;
\While{$U$ {\rm not empty}}
{
select one of the $i \in U$ that has a maximal $\lambda_i$\;
make $i$ a C-point ($C=C \cup i$, $U=U \setminus i$)\;
make all $j\in U$ that are strongly influenced by $i$, new F-points ($F=F \cup j$, $U=U \setminus j$)\;
increment $\lambda_k$ for all $k \in U$ that strongly influence the new F-points $j$\;
}
\label{alg:coarsening}
\end{algorithm}

We implement coarsening as follows.
For the $u$-variables we employ the standard one-pass Ruge-Stueben coarsening algorithm \cite{rs} (see \alg{coarsening}) on $N=A A^t$ using strength of connection condition
\begin{align}
\label{eq:strength}
{\rm variable} \ i {\rm \ is \ strongly} & \ {\rm influenced \ by \ variable} \ j \nonumber\\
&\Updownarrow \\
|n_{i,j}| \ge & \, \theta \, \sum_k \, |n_{i,k}| \nonumber
\end{align}
with $0<\theta<1$ a fixed strength parameter that may be chosen dependent on the problem. (The $(i,j)$ matrix element of $N$ is denoted by $n_{i,j}$.) 
For diagonally dominant PDE discretizations, strength is often determined relative to the largest off-diagonal element in row $i$, using condition $|n_{i,j}| \ge \, \theta \, \max_{k\ne i} \, |n_{i,k}|$. We, however, target a broader class of problem matrices, and opt for strength condition (\ref{eq:strength}), which is somewhat more general. Note, however, that the magnitude of strength parameter $\theta$ typically needs to be chosen differently in the two approaches.
For the $v$-variables, we determine strong connections in the same way, for matrix $A^t A$.
Once the strong connections are determined, coarsening can be performed: \alg{coarsening} is executed to determine sets of C-points and F-points for the $u$-variables and the $v$-variables.

In a next step, first for the $u$-variables, we determine, for each F-point $i$ in $F_u$, a coarse interpolatory set $C^i_u$ which contains all C-points (points in $C_u$) that strongly influence point $i$ according to condition (\ref{eq:strength}) in $A A^t$. The coarse interpolatory sets $C^i_v$ of the $v$-variable F-points are determined in the same way based on $A^t A$.
This defines the sparsity patterns of the interpolation operators $P$ and $Q$. 
We explain this for $P$, and it is analogous for $Q$. For each C-point in $C_u$ with fine-level index $i$, we let $\alpha(i)$ be the index of point $i$ on the coarse level. For all points $i$ in $C_u$, row $i$ in $P$ is zero, except for $p_{i,\alpha(i)}=1$. For all F-points $i$ in $F_u$ , row $i$ in $P$ is zero, except for matrix elements $p_{i,\alpha(j)}$ where $j$ is an element of $i$'s coarse interpolatory set $C^i_u$.

Basing coarsening of the $u$-variables and the $v$-variables on $A A^t$ and $A^t A$, respectively, can be motivated by the observation that, on the finest level, the left singular vectors are eigenvectors of $A A^t$, and the right  singular vectors are eigenvectors of $A^t A$. Moreover, $A A^t$ and $A^t A$ are symmetric matrices, and AMG was built for that type of matrices. In that sense, using $A A^t$ is a natural choice for measuring connection strength between $u$-variables. 
Also, forming $A A^t$ can be done in $O(m)$ (assuming $m \ge n$) time for large classes of sparse matrices, so it does not overly add to the cost of our method. Note also that we only use $A A^t$ for coarsening, and not in the rest of the algorithm, so there is no deterioration in terms of condition numbers, which is a major reason to avoid calculating the left singular vectors as the eigenvectors of $A A^t$, and the right singular vectors from $A^t A$). Note also that \eqs{svdGenSeparate2} suggest basing coarsening on $A^t \, B^{-1} \, A$ and $A \, C^{-1} \, A^t$ on coarser levels rather than $A^t A$ and $A A^t$, but we normally choose to ignore the $B^{-1}$ and $C^{-1}$ mass matrix factors to avoid the extra matrix inversion and matrix product. 

For some applications, however, it may be possible to devise good coarsening schemes for $u$ and $v$ directly from the rectangular matrix $A$, by considering its rows and columns. We expect, however, that the details and success of such strategies may be highly dependent on the type of problem, and direct coarsening methods for row-variables and column-variables of rectangular matrices is kept as an interesting topic of further research.

\subsection{Building $P$ and $Q$: Least-Squares Determination of Interpolation Weights}
We use a least-squares (LS) process to determine the interpolation weights in the rows of $P$ and $Q$ that correspond to F-points, following the approach in \cite{bamg,kushnir}. Again, we explain the process for matrix $P$, and it is analogous for $Q$.
We want to fit the interpolation weights of $P$ such that the $n_t$ current fine-level test vectors $U_t$ and the $n_b$ current boot vectors $U_b$ (except in the first cycle) lie approximately in the range of $P$. Let $U_f$ hold in its columns the $n_f=n_t+n_b$ vectors to be fitted. Let $u_k$ be the $k$th vector in $U_f$. Let $u_{k,c}$ be the coarse-level version of $u_k$ obtained by injection, and let $u_{k,c}^j$ be its value in coarse-level point $j$. Also, let $u_k^i$ be the value of $u_k$ in fine-level point $i$.
The weights of each F-point row in $P$ are determined consecutively using independent LS fits.
Consider a fixed F-point with fine-level index $i$ (the row index of $P$). Its coarse interpolatory set is $C^i_u$, and we assume now that the points in $C^i_u$ are labeled by their coarse-level indices (the column indices of $P$). Let $n_{c,i}$ be the number of elements of $C^i_u$.
For each F-point $i$ we solve the following least-squares problem to determine the unknown interpolation weights $p_{i,j}$:
\begin{align}
u_k^i=\sum_{j\in C^i_u} p_{i,j} u_{k,c}^j \quad (k=1,\ldots,n_f). 
\label{eq:LS}
\end{align}
This is a system of $n_f$ equations in $n_{c,i}$ unknowns.
We make this system overdetermined in all cases by choosing the number of initially random test vectors, $n_t$, larger than the expected largest interpolation stencil size $n_{c,i}$ for any $i$ on any level. (This is one of the criteria guiding the choice of $n_t$, and, in our implementation, estimating $n_t$ too small initially may require a restart of the method with a larger $n_t$).
Since we would like the dominant boot vectors to be fitted preferentially as soon as they become reasonable approximations, we weight the $k$th equation in \eq{LS} by the Rayleigh quotient, (\ref{eq:Rayleigh}), of the pair ($u_k$, $v_k$), see also \cite{kushnir}.
In our implementation, we solve the LS problem using a standard normal equation approach.
Finally, we mention that we use a modification of \eq{LS} for the case of minimal singular triplets or eigenpairs, as proposed in \cite{ibamg}. For these cases, interpolation weights and convergence can be improved significantly by applying an extra fine-level Jacobi relaxation (using the operator we base strength on) to the F-point values $u_k^i$ in \eq{LS} (but not the C-point values $u_k^j$), see \cite{ibamg} for further details. We have found in our numerical experiments that this modification is not useful when seeking dominant singular triplets or eigenpairs.
\section{AMG SVD Algorithm: Additive Phase}
\label{sec:add}
In the additive (solve) phase of our algorithm, we use fixed interpolation and coarse-level operators, namely, the operators $P$, $Q$, $A_c$, $B_c$ and $C_c$ as they were determined on all levels in the last mutiplicative cycle, and use an additive correction scheme to improve the $n_b$ boot singular triplets that came out of the multiplicative (setup) cycle at the finest level.
In each iteration of the additive phase, for each of the finest-level $\sigma_j$, $u_j$, $v_j$ ($1 \le j \le n_b$) in $\sigma_b$, $U_b$, $V_b$, we first improve $u_j$ and $v_j$ in a classical-type additive AMG V-cycle with $\sigma_j$ fixed in the whole cycle. Then, after all the $u_j$ and $v_j$ have been updated using one V-cycle for each pair, we collectively improve all the $\sigma_j$, $u_j$ and $v_j$ in $\sigma_b$, $U_b$, $V_b$ using a Ritz projection step on the finest level. These multigrid-Ritz iterations are repeated until the desired accuracy is reached. Our solve phase is similar to the approach described by Borzi and Borzi in \cite{borzi} for calculating minimal eigenpairs of an SPD matrix using standard AMG interpolation operators (it is also described in \cite{kushnir}, but not combined with a multiplicative phase). We now extend this approach to the calculation of dominant SVD triplets using the self-learned operators from the multiplicative phase of the algorithm.
\subsection{Coarse-level Equations}
In the additive correction scheme, the equations for triplet $(\sigma_j,u_j,v_j)$ on the current level are given by
\begin{alignat}{1}
&A \, v_j -\sigma_j \, B \, u_j = \kappa_j, \nonumber \\
&A^t\, u_j - \sigma_j \, C \, v_j = \tau_j,
\label{eq:add}
\end{alignat}
where $\kappa_j$ and $\tau_j$ are the residuals restricted down from the next finer level.
(So $\kappa_j=0$ and $\tau_j=0$ on the finest level.) 

The equations on the next coarser level are then
\begin{alignat}{1}
A_c \, v_{j,c} -\sigma_j \, B_c \, u_{j,c}= P^t \, r_j, \nonumber \\
A_c^t \, u_{j,c} -\sigma_j \, C_c \, v_{j,c} = Q^t \, s_j,
\label{eq:addCoarse}
\end{alignat}
where $r_j$ and $s_j$ are the residual vectors of the first and second fine-level equations, respectively,
and the coarse-grid correction equations are given by
\begin{alignat}{1}
&u^{(i+1)}_j=u^{(i)}_j+P\, u_{j,c},\nonumber \\
&v^{(i+1)}_j=v^{(i)}_j+Q\, v_{j,c},
\label{eq:corrAdd}
\end{alignat}
where the superscript $(i)$ means the $i$th iterate.
Note that $u_{j,c}$ and $v_{j,c}$ now represent coarse-level additive errors of the fine-level quantities $u_j$ and $v_j$.
Rather than using new variable names to distinguish original variables an their coarse-level errors, we follow the convention that is common in the multigrid literature \cite{mgtut} to refer to variables and their coarse-level errors with the same letter from the alphabet, which aids in presenting the algorithm in a recursive way.

Note that for the eigenvalue solvers in \cite{borzi, kushnir} the additive method is described in the framework of the full approximation scheme (FAS), like in the paper in which the general ideas of this approach were originally proposed \cite{mgeigen}, where the FAS framework was required because eigenvalue approximations were modified on the coarsest level of each cycle.
However, in the additive methods in \cite{borzi,kushnir}, eigenvalues remain fixed for the entire additive cycle, so there is no need to use the FAS, and the simpler error equation formulation that is common in multigrid for linear operators can be used instead, which is what we do in our discussion here.
\subsection{Additive V-cycles to Improve the Left and Right Singular Vectors}
For each of the finest-level $\sigma_j$, $u_j$, $v_j$ ($1 \le j \le n_b$) in $\sigma_b$, $U_b$, $V_b$, we fix $\sigma_j$
and perform an additive V-cycle as follows. We relax the singular vectors $u_j$ and $v_j$ using \eq{add} on the finest level, with the relaxation method that was described in \subsec{relaxBoot}. We calculate the residuals $\kappa_j$ and $\tau_j$, and restrict them to the next coarser level. We then choose a zero initial guess for $u_{j,c}$ and $v_{j,c}$ and relax them using coarse equations (\ref{eq:addCoarse}), we calculate the coarse residuals, restrict them to the next coarser level, etc. This is repeated recursively up to some coarse level where the problem is small enough for a direct solve. On the coarsest level, we solve \eq{addCoarse} exactly for vector $[u_{j,c}^t \, v_{j,c}^t]^t$ (as in \eq{big}). To make the coarse-level solve somewhat more robust when the operator is close to singular, one can optionally use the pseudo-inverse (calculated via the SVD) of $X-\sigma\,Y$ without including the component corresponding to its smallest singular value, as suggested in \cite{fly}. We do so in the numerical results presented in \sect{numres}.
We then interpolate the coarsest-grid solution up, correct using \eqs{corrAdd}, relax the corrected vectors, interpolate up again, etc., recursively until the finest level.
\subsection{Ritz Projection Step on the Finest Level to Improve the Boot Singular Triplets}
After carrying out one V-cycle for each of the $n_b$ boot singular triplets, we perform a Ritz projection step, as in \cite{borzi,kushnir}. An alternative would be to update each $\sigma_j$ in $\sigma_b$ using Rayleigh quotient formula (\ref{eq:Rayleigh}). However, a collective Ritz step leads to faster overall convergence, and has other important advantages. For singular values with multiplicity larger than one, it provides orthogonal singular vectors, and it precludes convergence of some of the triplets in the finest-level $\sigma_b$, $U_b$ and $V_b$ to spurious duplicate triplets, which may occur with the $\sigma$s updated individually according to \eq{Rayleigh}.

The Ritz step proceeds as follows.
We first orthogonalize the columns of $U_b$ with respect to $B$ using the QR decomposition,
and we orthogonalize the columns of $V_b$ with respect to $C$.
(Note that $B=I_m$ and $C=I_n$ on the finest level, but, in the multiplicative phase, the Ritz procedure can in principle also be employed on coarser levels, so we prefer to give the more general equations here.)
Let $\hat{U}$ and $\hat{V}$ be the orthogonalizations of $U_b$ and $V_b$, and 
let $\mU={\rm span}(\hat{U})$ and $\mV={\rm span}(\hat{V})$.
We seek new $u_j \in \mU$, $v_j \in \mV$, and $\sigma_j$ ($1 \le j \le n_b$) such that
\begin{alignat}{2}
& \left< u , A \, v_j - \sigma_j \, B \, u_j  \right>_B = 0 \quad & \forall u \in  \mU, \nonumber \\
& \left< v , A^t \, u_j - \sigma_j \, C \, v_j  \right>_C = 0 \quad & \forall v \in  \mV.
\label{eq:Ritz}
\end{alignat}
These equations express that the residuals are desired to be orthogonal ($B$-orthogonal and $C$-orthogonal, respectively) to the spaces $\mU$ and $\mV$ in which we seek an improved approximation.
\eq{Ritz} can be expressed in terms of new variables $y, y_j \in I\!\!R^{m_c}$ and 
$z, z_j \in I\!\!R^{n_c}$ with $u=\hat{U}\,y$, $v=\hat{V}\,z$, $u_j=\hat{U}\,y_j$ and $v_j=\hat{V}\,z_j$,
as
\begin{alignat}{2}
& \left< y , \hat{U}^t \, A \, \hat{V} \, z_j - \sigma_j \, \hat{U}^t \, B \, \hat{U} \, y_j  \right> = 0 \quad & \forall y \in I\!\!R^{m_c},  \nonumber \\
& \left< z , \hat{V}^t \, A^t \, \hat{U} \, y_j - \sigma_j \, \hat{V}^t \, C \, \hat{V} \, z_j  \right> = 0 \quad & \forall z \in  I\!\!R^{n_c}.
\label{eq:RitzSmall}
\end{alignat}
The following generalized eigenvalue problem of size $2\, n_b \times 2\, n_b$ results
\begin{align}
\left(
\left[
\begin{array}{cc}
0 & \hat{U}^t \, A \, \hat{V}\\
\hat{V}^t \, A^t \, \hat{U} & 0
\end{array}
\right]
-\sigma_j \,
\left[
\begin{array}{cc}
\hat{U}^t \, B \, \hat{U} & 0\\
0 &  \hat{V}^t \, C \, \hat{V}
\end{array}
\right]
\right) \,
\left[
\begin{array}{c}
y_j\\
z_j
\end{array}
\right]
=0.
\label{eq:genSVDRitz}
\end{align}
According to \thm{genSVD}, the eigenvalues of \eq{genSVDRitz} occur in pairs symmetrically about zero, and it is sufficient to consider the $n_b$ triplets $(\sigma_j,y_j,z_j)$ with the largest values for $\sigma_j$ to generate new approximations $(\sigma_j,\hat{U} \, y_j, \hat{V} \, z_j)$
for the dominant singular triplets on the finest level.

Note finally that, unlike the multiplicative cycles, the multigrid-Ritz additive iterations can converge to any required accuracy, even though, on each level, the $u_j$ are not exactly in the range of the $P$s, and the $v_j$s are not exactly in the range of the $Q$s. In practice, as demonstrated in the numerical tests below, the hybrid multiplicative-additive scheme converges up to machine accuracy if desired.
\section{AMG SVD Algorithm: Specialization and Extension}
\label{sec:spec}
In this section we discuss the specialization of the dominant singular triplet algorithm for rectangular matrices to the case of square matrices and symmetric matrices (dominant eigenpairs), and its extension to the case of minimal singular triplets (and eigenpairs).

\subsection{Singular Triplets of Square Matrices}
A possible simplification for square, nonsymmetric matrices is that interpolation operators $P$ and $Q$ could potentially be based on $A$ and/or $A^t$; it does not appear to be necessary to form $A A^t$ and $A^t A$, so that cost may be saved. Interestingly, if one wants to keep square matrices on all levels, coarsening and sparsity patterns for $P$ and $Q$ should both be based on either $A$ or $A^t$, because coarsening of $A$ and $A^t$ may lead to different numbers of coarse grid points (except if a coarsening method is used that is symmetric).
If the left and right singular vectors are expected to be very similar such that they can all be fitted with reasonable accuracy by one interpolation operator, $P$ and $Q$ could even be taken the same on all levels ; in that case it would also hold that $B_c=C_c$ on all levels, which can be exploited for further cost savings.
\subsection{Eigenpairs of Symmetric Matrices}
In the case of symmetric matrices, the whole algorithm simplifies significantly, and becomes a combination of the minimal SPD eigenpair algorithms of \cite{borzi} and \cite{kushnir}, extended to dominant eigenpairs.
The resulting algorithm can be formulated in terms of operators $A$, $B$ and $P$ on all levels.
This combination of a multiplicative and an additive scheme into a hybrid method for eigenpairs
has the advantages that it can converge up to machine accuracy for multiple eigenvectors with one $P$, and that it is self-learning.
\subsection{Minimal Singular Triplets and Minimal Eigenpairs}
With just a few small modifications, the hybrid multiplicative-additive dominant singular triplet algorithm described above can also be used to compute the $n_b$ singular triplets with smallest singular values. All that is required is to modify the relaxation schemes, and to select the smallest singular triplets as new boot singular triplets in the coarsest-level solve of the multiplicative phase. The weights in the LS fitting of the test and boot vectors is taken as the inverse of the Rayleigh quotient, see also
\cite{kushnir,bamgQcd,bamgMarkov}.
For the relaxation of the $n_t$ initially random test vectors in $U_t$ and $V_t$, we iterate on 
\eqs{svdGenSeparate} with $\sigma=0$ using Kaczmarz relaxation
(see \cite{tanabe,kushnir}). Richardson iteration as in \cite{nonsym} can be considered as another option for relaxation.
For the relaxation of the $n_b$ boot vectors in $U_b$ and $V_b$, we iterate on \eqs{genSVDresid1}-\noeq{genSVDresid2} (with the small $\sigma$s from $\sigma_b$) in a block GS fashion using Kaczmarz relaxation \cite{tanabe,kushnir} for the blocks.
Numerical tests show that these Kaczmarz relaxations may sometimes result in singular vector pairs that produce a negative Rayleigh quotient. We test for this and reverse the sign of one of the singular vectors if this happens.
In the case of minimal eigenpairs of symmetric matrices, GS relaxation on $A\,x=0$ can be used, with Kaczmarz on coarser levels, see \cite{kushnir}. In the numerical results reported below, we use Kaczmarz relaxation on all levels when seeking minimal singular triplets or eigenpairs. Note also that, since our method is self-learning, the minimal SPD eigenpair problem can in principle also be solved simply by shifting the operator such that the spectrum ends up at the other side of the origin, and then the algorithm for dominant eigenpairs can be used (and vice versa).

\section{Numerical results}
\label{sec:numres}

In this section, we present numerical results illustrating how our proposed method performs. We discuss four different test problems that cover the different cases of rectangular matrices, square nonsymmetric matrices, and symmetric matrices.

\subsection{High-Order Finite Volume Element Laplacian on Unit Square}
\label{subsec:FVE}
In the first test problem, we seek a few extremal singular triplets of a square, nonsymmetric matrix that results from a finite volume element (FVE) discretization with quadratic polynomials of the standard Laplacian operator on the unit square with Dirichlet boundary conditions, see \cite{vogel,fve}. The operator is discretized on a structured triangular grid. For this problem, the FVE method with linear polynomials gives a discretization that is exactly the same as the Galerkin finite element discretization with linear polynomials. For higher orders, however, the FVE discretization is slightly non-symmetric.

\begin{figure}[!htbp]
  \centering
  \scalebox{0.5}{\includegraphics{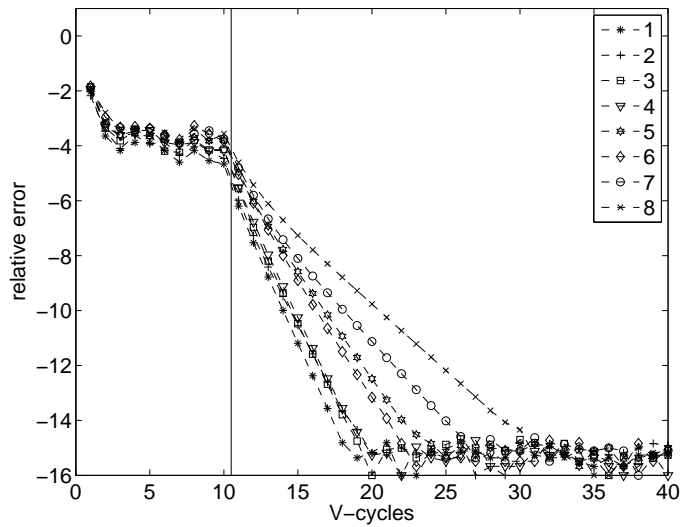}}
   \caption{Largest Singular Values for High-Order Finite Volume Element Laplacian on Unit Square (square, nonsymmetric matrix). Convergence plot for calculation of the eight largest singular values (base-10 logarithm of relative error in singular values as a function of number of V-cycle iterations). Singular values are labeled with decreasing magnitude (label 1 denotes the largest singular value). The 10 V-cycles to the left of the vertical line are multiplicative, and the 30 V-cycles to the right of the vertical line are additive.} 
\label{fig:FVElg}
\end{figure}    

Figs.\ \ref{fig:FVElg} and \ref{fig:FVEsm} show convergence results for approximating the largest and smallest singular values, respectively, for a matrix with $m=n=961$ ($31\times31$ internal grid points). We show the base-10 logarithm of the relative error in the calculated singular values
\begin{align}
error=\frac{|\sigma_{exact}-\sigma_{approx}|}{\sigma_{exact}},
\label{eq:error}
\end{align}
as a function of the number of V-cycles. Here, the values $\sigma_{exact}$ are high-accuracy approximations obtained by Matlab's built-in SVD algorithms. There are 10 multiplicative (setup) cycles followed by 30 additive (solve) cycles.
We have calculated $n_b=8$ dominant or minimal singular triplets, using $n_t=5$ initially random test vectors. We used 
$\mu_t=4$ relaxations on the test vectors, and $\mu_b=4$ relaxations on the boot vectors, on all levels. The coarsening strength parameter was chosen as $\theta=0.05$. Coarsening and sparsity patterns for both $P$ and $Q$ are determined using $A$, thus guaranteeing square matrices $A$ on all levels.

The figures show that the extremal singular triplet algorithm carries out the task that is was designed for: it collectively calculates several singular values up to machine accuracy in a modest number of multigrid V-cycles, and this both for the dominant triplet and the minimal triplet case. The initial, multiplicative phase approximately determines singular triplets starting from initially random test vectors, but convergence stagnates after a few operations because it is limited by the accuracy by which the singular vectors are represented collectively by single interpolation operators. A second, additive phase succeeds in driving the error to machine accuracy, using the (fixed) interpolation operators that were derived in the last multiplicative iteration. This shows that the approach is able to fit interpolation to the relevant vectors both for the cases of dominant and minimal triplets.

For conciseness, we will limit ourselves to plot the relative errors in singular values or eigenvalues in this paper. Convergence of these properties goes along with high-accuracy convergence of other quantities like residuals, angles between exact and approximate singular vectors, orthogonality measures between singular vectors, etc. All these quantities also converge with high accuracy in our numerical tests, but they are not shown for conciseness.
Since our code is implemented in Matlab and is not optimized, we do not directly compare with other, optimized codes in terms of CPU time, but instead focus on reporting convergence numbers as a function of numbers of V-cycle iterations, which gives valuable insight in the effectiveness of our method, since the cost of a V-cycle is approximately linear in the number of unknowns, $m+n$.

For the case of dominant singular triplets (\fig{FVElg}), the calculation uses four levels, with coarsest size $45\times45$. For the case of minimal singular triplets, five levels were obtained, with a coarsest grid of size $51\times51$.
See \tabl{values} for approximations of the singular values calculated. It can be seen that the singular values lie very close to each other, which makes this a difficult type of problem for many iterative singular value decomposition algorithms. Nevertheless, our algorithm converges to machine accuracy in a moderate number of V-cycles.
Note also that the non-symmetry of the discrete operator has lifted the degeneracy of the continuous operator, which has eigenvalues with multiplicity larger than one; no singular values with multiplicity larger than one arise.
                                                                                         
\begin{figure}[!htbp]
  \centering
  \scalebox{0.5}{\includegraphics{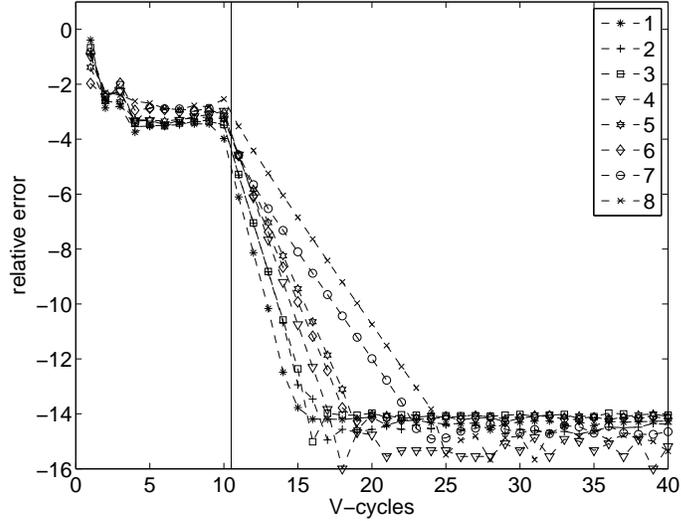}}
   \caption{Smallest Singular Values for High-Order Finite Volume Element Laplacian on Unit Square (square, nonsymmetric matrix). Convergence plot for calculation of the eight smallest singular values (base-10 logarithm of relative error in singular values as a function of number of V-cycle iterations). Singular values are labeled with increasing magnitude (label 1 denotes the smallest singular value). The 10 V-cycles to the left of the vertical line are multiplicative, and the 30 V-cycles to the right of the vertical line are additive.} 
\label{fig:FVEsm}
\end{figure}    

\begin{table}[h!]
    \begin{center}
        \begin{tabular}{|l|l|l|l|l|l|l|}\hline
FVE lge&FVE sm&FD lge&FD sm&Graph lge&Graph sm&Term-Doc \\ \hline
7.9791546 & 0.01924183&7.9818877&0.01811231&13.509036&0.01000000&84.148337 \\ \hline
7.9491729& 0.04794913&7.9548012&0.04519876&13.352613&0.03456116&64.707532 \\ \hline
7.9468326& 0.04801773&7.9548012&0.04519876&13.350454&0.03901593&55.976437 \\ \hline
7.9172573& 0.07655365&7.9277148&0.07228521&12.472837&0.07966567&50.265499 \\ \hline
7.8965349& 0.09557904&7.9099298&0.09007021&12.416200&0.09490793&49.265360 \\ \hline
7.8960066& 0.09558103&7.9099298&0.09007021&11.874669&0.09918138&45.242034 \\ \hline
7.8692955& 0.12359047&7.8828433&0.11715666&&&44.400811 \\ \hline
7.8616683& 0.12415144&7.8828433&0.11715666&&&41.772394 \\ \hline
        \end{tabular}
    \end{center}
    \caption{Singular values and eigenvalues sought for each problem (high-accuracy approximations).}
    \label{tab:values}
\end{table}

\begin{figure}[!htbp]
  \centering
  \scalebox{0.5}{\includegraphics{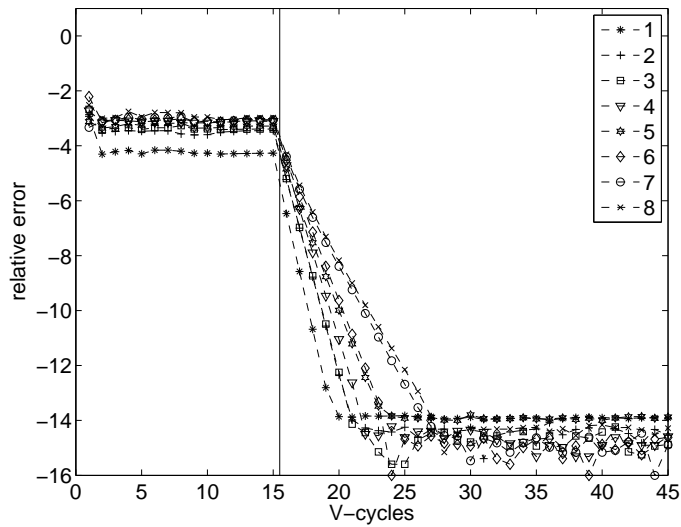}}
   \caption{Smallest Eigenvalues for Finite Difference Laplacian on Unit Square (square, symmetric matrix). Convergence plot for calculation of the eight smallest eigenvalues (base-10 logarithm relative error in eigenvalues as a function of number of V-cycle iterations). Eigenvalues are labeled with increasing magnitude (label 1 denotes the smallest eigenvalue). The 15 V-cycles to the left of the vertical line are multiplicative, and the 30 V-cycles to the right of the vertical line are additive.} 
\label{fig:FDsm}
\end{figure}    
  
\subsection{Finite Difference Laplacian on Unit Square}
We now consider the case of a simple finite-difference (FD) Laplacian with Dirichlet boundary conditions discretized with a 5-point stencil on a unit square with a Cartesian grid. This leads to a symmetric matrix (it is SPD), and we seek minimal and dominant eigenpairs. We use strength of connection $\theta=0.06$ and seek $n_b=8$ minimal or dominant eigenpairs, using $n_t=6$ initially random test vectors. We used $\mu_t=8$ relaxations on the test vectors, and $\mu_b=4$ relaxations on the boot vectors. We perform 15 multiplicative cycles followed by 30 additive cycles. The problem size is $m=n=1024$ ($32\times32$ internal grid points). \tabl{values} shows that there are eigenvalues with multiplicity larger than one for this symmetric discretization.

\fig{FDsm} shows convergence results for the case of minimal eigenpairs. Five levels are used and the coarsest grid is of size $64\times 64$. These results can be compared with the results of the additive-only eigenvalue method of Borzi and Borzi (\cite{borzi}) and the multiplicative-only eigenvalue method of Kushnir, Galun and Brandt (\cite{kushnir}). 
Our additive phase is like the method in \cite{borzi}, but in that paper standard AMG interpolation is used. We appear to get similar results, but our method is more general and can also be applied to seeking dominant eigenpairs and to a wider set of problems due to its self-learning capacity. Our multiplicative phase is like the method in \cite{kushnir}. We see that convergence stagnates at the level of accuracy by which interpolation collectively represents the desired eigenvectors. (Note that in our combined algorithm it would have been sufficient to perform less than 15 multiplicative cycles.) In \cite{kushnir} interpolation is made more accurate to improve the accuracy level at which the collective multiplicative phase stagnates. As explained in that paper, the accuracy that can be obtained in this way may be sufficient for some applications, for example, due to unavoidable discretization errors in PDE problems, or due to data and model uncertainties in data analysis tasks. In our approach, we show that, if desired, higher accuracy can be obtained by combining the multiplicative and additive approaches, resulting in a method that is flexible enough to deal efficiently with a variety of problems due to its self-learning capabilities.
\begin{figure}[!htbp]
  \centering
  \scalebox{0.5}{\includegraphics{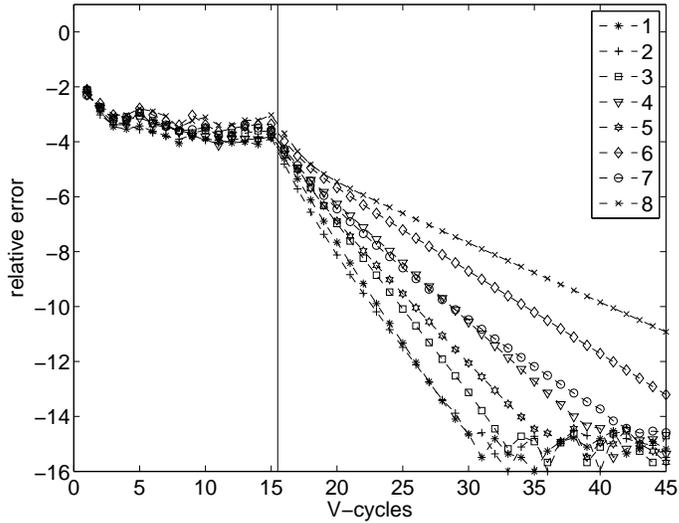}}
   \caption{Largest Eigenvalues for Finite Difference Laplacian on Unit Square (square, symmetric matrix). Convergence plot for calculation of the eight largest eigenvalues (base-10 logarithm of relative error in eigenvalues as a function of number of V-cycle iterations). Eigenvalues are labeled with decreasing magnitude (label 1 denotes the largest eigenvalue). The 15 V-cycles to the left of the vertical line are multiplicative, and the 30 V-cycles to the right of the vertical line are additive.} 
\label{fig:FDlg}
\end{figure}    
\fig{FDlg} gives convergence results for the case of dominant eigenpairs. Four levels are used and the coarsest grid is of size $52\times52$. The results show that our hybrid multiplicative-additive method can also compute dominant eigenpairs, extending the approaches for minimal eigenpairs from \cite{borzi,kushnir} to dominant eigenpairs. Convergence in the additive phase appears somewhat slower than for the minimal eigenpairs case. This may be due to the fact that we employ Kaczmarz relaxation for the minimal eigenpairs, which is more efficient but also more expensive than the inexact power method relaxation used for the dominant eigenpairs case (\subsec{relaxBoot}).
It is interesting to note that the approach in \cite{borzi} which uses standard AMG interpolation, can also be extended to calculating dominant eigenpairs simply by changing the signs of all off-diagonal interpolation weights. The resulting interpolation operators turn out to be good fits for the most oscillatory modes, and can be used in an additive scheme to approximate the dominant eigenpairs.

\subsection{Planar Random Triangulation Graph Laplacian}
\begin{figure}[!htbp]
  \centering
  \scalebox{0.5}{\includegraphics{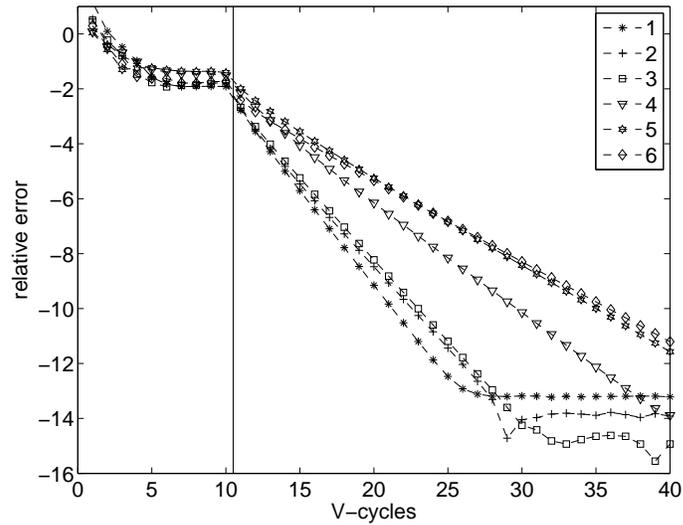}}
   \caption{Smallest Eigenvalues for Planar Random Triangulation Graph Laplacian (square, symmetric matrix). Convergence plot for calculation of the six smallest eigenvalues (base-10 logarithm of relative error in eigenvalues as a function of number of V-cycle iterations). Eigenvalues are labeled with increasing magnitude (label 1 denotes the smallest eigenvalue). The 10 V-cycles to the left of the vertical line are multiplicative, and the 30 V-cycles to the right of the vertical line are additive.} 
\label{fig:Graphsm}
\end{figure}    
\begin{figure}[!htbp]
  \centering
  \scalebox{0.5}{\includegraphics{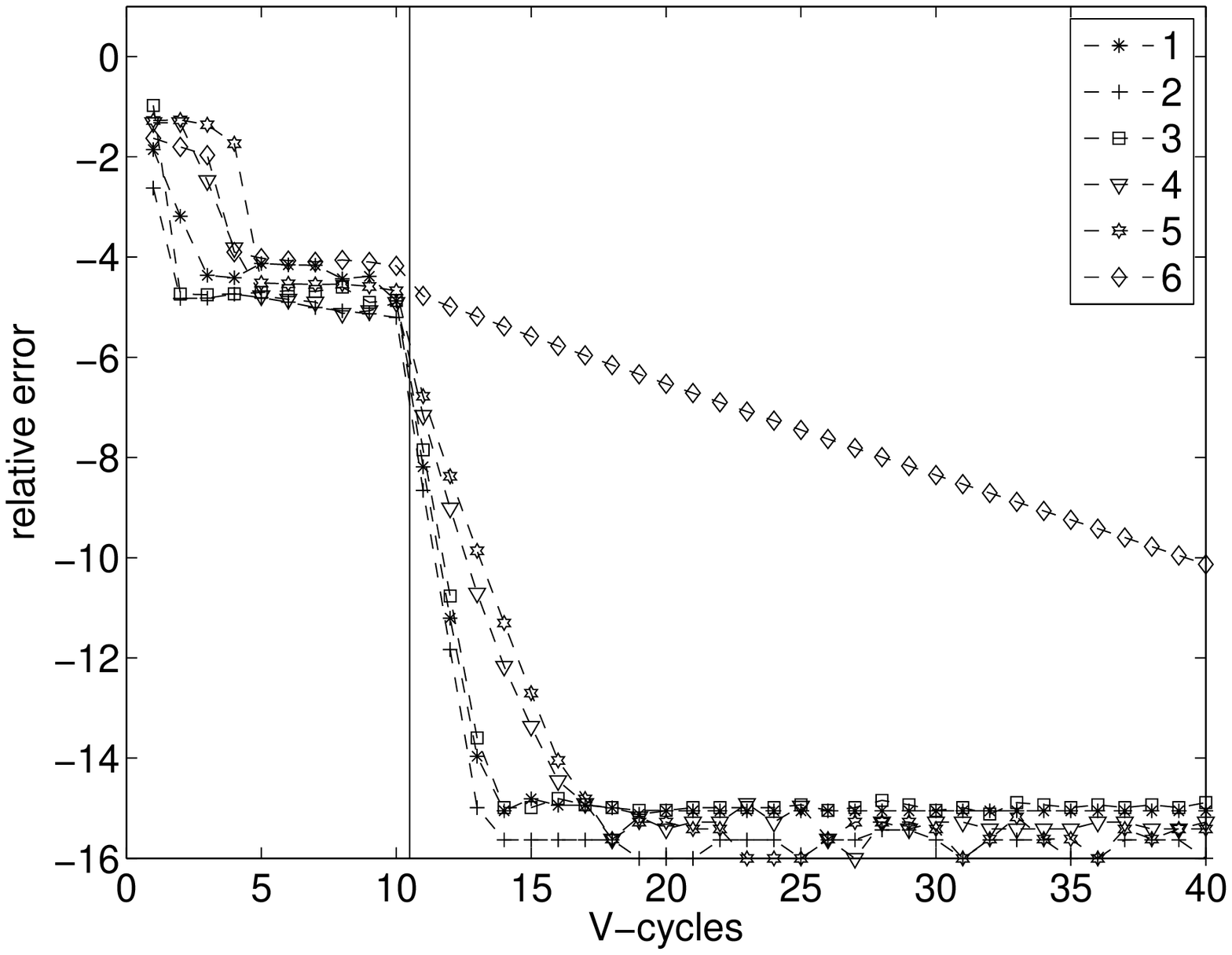}}
   \caption{Largest Eigenvalues for Planar Random Triangulation Graph Laplacian (square, symmetric matrix). Convergence plot for calculation of the six largest eigenvalues (base-10 logarithm of relative error in eigenvalues as a function of number of V-cycle iterations). Eigenvalues are labeled with decreasing magnitude (label 1 denotes the largest eigenvalue). The 10 V-cycles to the left of the vertical line are multiplicative, and the 30 V-cycles to the right of the vertical line are additive.} 
\label{fig:Graphlg}
\end{figure}    
\begin{figure}[!htbp]
  \centering
  \scalebox{0.5}{\includegraphics{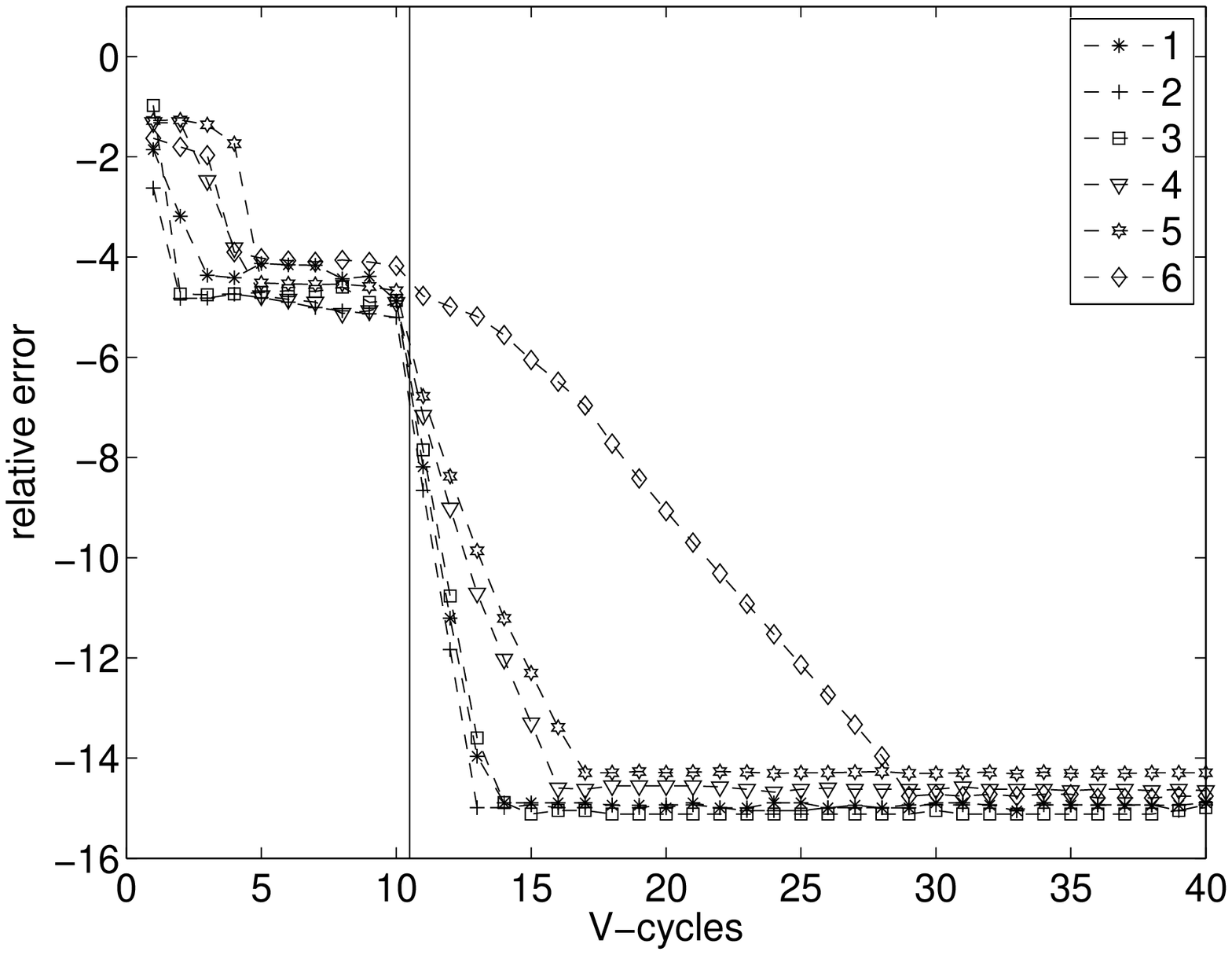}}
   \caption{Same as \fig{Graphlg}, but during the additive phase, whenever one or more of the eigenvalues reach a relative error converge tolerance of 1e-14, the interpolation operators are redetermined and preferentially fitted to the unconverged eigenpairs. This improves the convergence of the eigenpair that is slow to converge in \fig{Graphlg}.} 
\label{fig:GraphlgRedoP}
\end{figure}    

The next test problem is the graph Laplacian operator of a planar random graph that is obtained by placing points uniformly random in the unit square and determining their Delauney triangulation graph. With $\hat{A}$ the adjacency matrix of the graph, the graph Laplacian, $A$, can be constructed by setting $A=-\hat{A}$ and placing the row sums of $\hat{A}$ on the diagonal. This results in a symmetric semi-definite matrix (it has one vanishing eigenvalue), and we seek dominant and minimal eigenpairs. This problem is interesting as a test case because it is unstructured, contrary to the previous two problems which derive from structured grids. Graph Laplacian matrices are of interest in data analysis tasks \cite{kushnir}. We use strength of connection $\theta=0.05$ and seek $n_b=6$ dominant or minimal eigenpairs, using $n_t=6$ initially random test vectors. We use $\mu_t=1$ relaxations on the test vectors, and $\mu_b=8$ relaxations on the boot vectors. We perform 10 multiplicative cycles followed by 30 additive cycles. The problem size is $m=n=1024$.

\fig{Graphsm} shows convergence results for the case of minimal eigenpairs. Three levels are used and the coarsest grid is of size $59\times 59$. The operator is shifted by 0.01 to avoid problems in representing the relative error in the smallest eigenvalue (which vanishes for the unshifted operator). \fig{Graphsm} shows satisfactory convergence behavior, but convergence in the additive phase is not as good as for the finite difference Laplacian on a structured grid (\fig{FDsm}), even though we doubled $\mu_b$ to 8. This is most likely due to the fact that the minimal eigenvectors of the unstructured problem are less regular and less similar to each other, such that they are not represented as well by the collective interpolation operators.
For this reason, we only sought six eigenpairs for this problem.
We reduced the number of test vector relaxations because the eigenvalues are less clustered for this problem, and too many test vector relaxations quickly make the set of test vectors too linearly independent for the LS fits. \fig{Graphlg} gives convergence results for the case of dominant eigenpairs. Three levels are used and the coarsest grid is of size $77\times77$. It can be seen that the algorithm converges slowly for the sixth eigenpair. When one or more of the eigenpairs sought converge significantly more slowly than the others, the following strategy can be followed to improve their convergence. In the additive phase, once some eigenpairs have converged beyond a pre-specified tolerance, one can redetermine the interpolation operators in a way to preferentially fit the eigenpairs that have not converged yet. \fig{GraphlgRedoP} shows that this can improve the convergence of lagging eigenpairs. For the convergence curves shown in \fig{GraphlgRedoP}, whenever one or more of the singular values reach a relative error converge tolerance of 1e-14, we redetermine the interpolation operators (basically, by executing one downward sweep of the multiplicative phase), and reduce the weight of the already converged boot vectors and the test vectors by a factor of 1\,000 in the LS fitting process. This can speed up the convergence of the remaining eigenpairs, as shown in \fig{GraphlgRedoP}.

\subsection{Medline Term-document Matrix}
The final test matrix is a real term-document matrix, namely, the MEDLINE data set downloaded from the Text to Matrix Generator website (http://scgroup20.ceid.upatras.gr:8000/tmg). The rows of this matrix represent terms and the columns represent documents. Matrix element $(i,j)$ counts how many times term $i$ occurs in document $j$. The matrix is sparse (less than 1\% nonzeros). Latent semantic indexing determines concepts in documents by calculating dominant singular triplets of term-document matrices \cite{lsi}, so we seek to compute dominant singular triplets. We consider a rectangular submatrix of size $m=5\,735$, $n=1\,033$.
We use strength of connection $\theta=0.03$ and seek $n_b=8$ dominant singular triplets, using $n_t=14$ initially random test vectors. We used $\mu_t=1$ relaxations on the test vectors, and $\mu_b=4$ relaxations on the boot vectors. We perform 3 multiplicative cycles followed by 30 additive cycles.

\fig{TDlg} shows convergence results for approximating the eight dominant singular triplets. The calculation uses five levels, and the coarsest grid is of size $415\times198$. The figure shows that our method is successful in calculating the eight dominant singular triplets, with good convergence and high accuracy. The importance of this proof-of-concept calculation is that it indicates that our approach is flexible enough to deal with this kind of problem that is new to multigrid (as far as we are aware). The self-learning feature of our method is able to adapt to the singular vectors that are relevant in this application, which is interesting by itself, since our development is an extension of algebraic multigrid concepts that were developed for PDEs, in which the relevant vectors are of a different nature. Similarly, we have obtained the result in \fig{TDlg} using a standard PDE-oriented AMG coarsening approach, and obtain results that appear to converge quite satisfactorily. It has to be noted, though, that the dominant singular values of term-document matrices may have larger gaps (see \tabl{values}), especially for the very largest ones, which may make these problems somewhat easier for iterative methods than, for example, the FVE problem of \subsec{FVE}, which has small gaps between the dominant (and minimal) singular values that decrease with increasing problem size. While we expect our method to be competitive for the latter type of problems, it remains to be investigated in future work how competitive our general approach can be made for problems like term-document matrices. For one, it would require to consider dedicated special-purpose coarsening methods. (We have already developed such special-purpose coarsening mechanisms for certain scale-free graphs, see \cite{leaf}, and see also
\cite{cr,relcoarse} for promising more general approaches.) In the case of rectangular matrices, it may be possible to come up with methods to coarsen the row and column variables based on $A$ and $A^t$ directly (rather than using $A\,A^t$ and $A^t\,A$), which may be feasible for some applications, guided by the application-dependent interpretation of the variables and operator matrix coefficients, and is kept for future work.
Nevertheless, the proof-of-concept results presented here already show promise and illustrate the versatility of our general approach to calculating singular triplets.
\begin{figure}[!htbp]
  \centering
  \scalebox{0.5}{\includegraphics{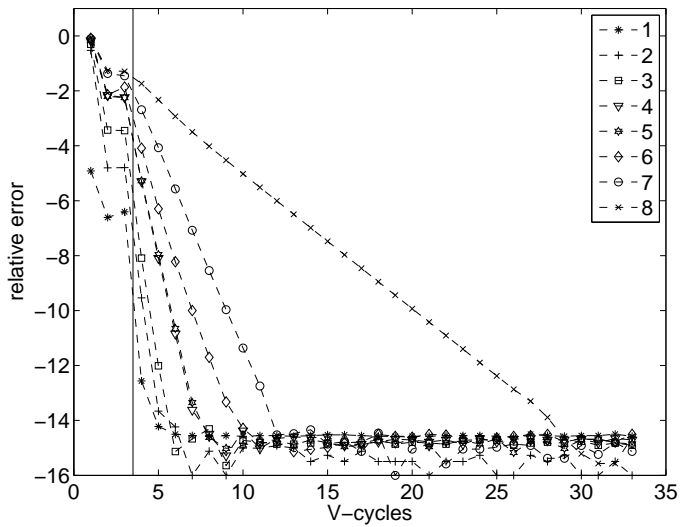}}
   \caption{Largest Singular Values for Medline Term-document Matrix (rectangular). Convergence plot for calculation of the eight largest singular values (base-10 logarithm of relative error in singular values as a function of number of V-cycle iterations). Singular values are labeled with decreasing magnitude (label 1 denotes the largest singular value). The 3 V-cycles to the left of the vertical line are multiplicative, and the 30 V-cycles to the right of the vertical line are additive.} 
\label{fig:TDlg}
\end{figure}    

\subsection{Discussion}
The above numerical results show that the proposed combined multiplicative-additive approach is successful in calculating extremal singular triplets and eigenpairs, with high accuracy obtained in a modest number of V-cycles for a variety of problems. However, more research needs to be done to make the method more black-box and robust. There are quite a few parameters to be chosen, and success is sometimes sensitive to careful choice of these parameters. This is not unlike the situation that existed for AMG as a linear system solver early on in its development for that purpose; it took many years of concerted effort for AMG to ripen to the successful linear system solver technology that it is today, and self-learning AMG eigensolvers and singular triplet solvers will require time and effort to mature as well.
In addition, new types of application problems often require at least some modification in algorithmic components like coarsening schemes. In this sense, the present paper is still an early step in the development of AMG methods for singular triplets: it presents a framework and one particular way to implement the components, but these components have to be further improved and there are alternative candidates for some of them. 
For example, in the multiplicative phase, it is not always easy to find a good choice for the number of relaxations to be done on the test vectors. Too many relaxations may lead to linear dependence (and how many is too many depends on the a priori not necessarily known gaps in the extrema of the spectrum), and not enough relaxations may lead to coarse-level problems that do not identify the correct singular triplets. Similarly, the choice of the weight factors in the LS fitting is also not straightforward and results may depend on it significantly. These aspects need to be improved. Similarly, in the additive phase, the V-cycles may not convergence for some of the tentative triplets, and there is no guarantee that no triplets are missed (even though we have only rarely observed this).
Also, it would be interesting to consider special-purpose coarsening routines, for example, for principal component analysis data sets. For some applications, one may need mutiple $P$ and $Q$ interpolation operators to fit groups of (possibly overlapping) triplets, or coarse grids with multiple degrees of freedom per coarse grid point may need to be considered, especially if singular triplets have singular vectors that are very dissimilar.
In the mutiplicative phase, instead of using the BAMG approach, one could consider building up interpolation operators that fit the relevant vectors by using the so-called `adaptive' approach from \cite{asa,aamg}, and possibly extending it to fit multiple target vectors. Similarly, the current multigrid-Ritz additive phase could be replaced by methods of preconditioned inverse iteration, locally optimal block preconditioned conjugate gradient, or Rayleigh quotient multigrid type \cite{borzi,lobpcg,hetmaniuk}. Also, compatible relaxation processes may be considered for coarsening \cite{cr,relcoarse}. The results presented in this paper show initial success and promise for our general approach, but improvements and modifications of the components are possible and are a topic of continued research.
\section{Conclusion}
\label{sec:conc}
We have described a new algebraic multilevel framework for computing dominant and minimal singular triplets. As far as we are aware, this is the first algebraic multigrid method that directly tackles the SVD problem, without working on $A^t\,A$ or the augmented symmetric system. We combine a multiplicative phase with an additive phase to obtain a self-learning method that can converge to machine accuracy for multiple singular vectors represented collectively using single interpolation operators. The self-learning capability of the algorithm makes it applicable to many types of problems, both for dominant and minimal triplets.
We have identified a generalized SVD decomposition of a matrix $A$ relative to two SPD matrices $B$ and $C$ of compatible dimensions as the problem to be solved on the coarse levels of our multilevel method, and have stated its existence and uniqueness properties and discussed relevant solution methods.
Our multiplicative phase follows the BAMG framework, as in \cite{kushnir} for SPD eigenproblems, and our additive phase follows a multigrid-Ritz strategy, as in \cite{borzi} for SPD eigenproblems. The specialization of our combined method to SPD matrices offers a new extension of those existing AMG eigensolvers, that allows for highly accurate convergence and is flexible due to its self-learning nature.
Ongoing work is aimed at improving the parameter-independence and robustness of components of the algorithm, and alternative building blocks can be considered \cite{asa,aamg,borzi,lobpcg,hetmaniuk,cr,relcoarse} for some of the components in the algorithmic framework.
Numerical tests using our current implementation showed that convergence to high accuracy can be obtained in a modest number of V-cycles, and the versatility of the approach was illustrated by applying it to problems from different domains.


\end{document}